\documentclass[11pt]{amsart}
\usepackage{amsmath, amsfonts, amsbsy, amssymb,  enumerate, chatterjee1, upref}

\newcommand{\ep}{\epsilon}

\begin{document}
\title{The missing log in large deviations for triangle  counts}
%\title[Taming the infamous upper tail]{Taming the infamous upper tail: sharp large deviations for subgraph counts}
\author{Sourav Chatterjee}
\address{Courant Institute of Mathematical Sciences, New York University, 251 Mercer Street, New York, NY 10012}
\thanks{Research partially supported by NSF grants DMS-0707054 and  DMS-1005312,   and a Sloan Research Fellowship}
\keywords{Erd\H{o}s-R\'enyi random graph, tail bound, concentration inequalities, large deviations}

\begin{abstract}
This paper solves the problem of sharp large deviation estimates for the upper tail of the number of triangles in an Erd\H{o}s-R\'enyi random graph, by establishing a logarithmic factor in the exponent that was missing till now. It is possible that the method of proof may extend to general subgraph counts.
\end{abstract}

\maketitle
\section{Introduction}
Let $G(n,p)$ be an Erd\H{o}s-R\'enyi graph on $n$ edges with edge probability  $p$, that is, a random graph on $n$ vertices where each edge is added independently with probability $p$. Given a fixed small graph $H$, let $X_H$ denote the number of copies of $H$ in $G(n,p)$. The distribution of $X_H$ has been studied extensively since the work of Erd\H{o}s and R\'enyi \cite{erdosrenyi60} in 1960, where the first results were given. In the very sparse case, where only a few copies of $H$ can occur, the probability of $\{X_H > 0\}$ was studied by Bollob\'as~\cite{bollobas81}. A necessary and sufficient condition for the asymptotic normality of $X_H$ (when $n \ra \infty$ and $p$ remains fixed or $p\ra 0$) was obtained by Ruci\'nski \cite{rucinski88}. Sharp large deviation inequalities for $\pp(X_H \le (1-\ep)\ee(X_H))$, where $\ep$ is a fixed positive number were obtained by Janson, \L uczak and Ruci\'nski \cite{JLR00} via the method of Janson's inequality \cite{janson90}. 

One key question that remained open for a long time was the issue of sharp large deviations for the `upper tail', i.e.\ competent bounds for 
\[
\pp(X_H \ge (1+\ep)\ee(X_H)),
\]
where $\ep$ is a fixed constant. 
For historical accounts of this problem, one may look in \cite{bollobas85, JLR00, JR02}. The problem was almost completely intractable until the year~2000, when the first general exponential tail bound was obtained by Vu \cite{vu01}. However, Vu's upper bound from \cite{vu01}, as well as the upper bounds obtained by Janson and Ruci\'nski \cite{JR02} soon after, were quite far from the conjectured bounds. In fact, the paper \cite{JR02} surveys an exhaustive array of techniques, ranging from the simple Azuma-Hoeffding inequality~\cite{hoeffding63}, to the powerful inequalities of Talagrand~\cite{talagrand95}, in the context of this problem. But none of them yield anything close to what are believed to be the optimal bounds. Similarly, the striking developments of Boucheron, Lugosi and Massart \cite{blm03}, while successful in a wide range of concentration problems, could not quite resolve the issue of the upper tail for subgraph counts.

The major breakthroughs in the `upper tail problem' came with the works of Kim and Vu \cite{kimvu04} and Janson, Oleszkiewicz and Ruci\'nski \cite{JOR04} in 2004. Kim and Vu \cite{kimvu04} showed that if $H$ is a triangle, then for any $\ep >0$, there are positive constants $C_1(\ep)$ and $C_2(\ep)$ such that whenever $p \ge n^{-1}\log n$, 
\[
e^{-C_1(\ep)n^2 p^2 \log (1/p)} \le \pp(X_H \ge (1+\ep)\ee(X_H)) \le e^{-C_2(\ep) n^2 p^2}. 
\]
At the same time, Janson et.~al.~\cite{JOR04} proved a similar result for general $H$, with a difference of $\log (1/p)$ between the upper and lower bounds. (Note the lower bound is easily obtained as the probability of a single clique containing all the extra triangles.)

Kim and Vu's technique is based on a method developed in their earlier works \cite{kimvu00, vu02}, which in turn is a highly sophisticated version of the method of martingale differences~\cite{mcdiarmid89}. The method of Janson et.~al.\ builds on an extension of a remarkable result of Alon \cite{alon81} about the maximum number of copies of a graph $H$ in a graph with a given number of edges.

Subsequently, improvements in certain regimes were obtained by Janson and Ruci\'nski \cite{JR04}, but the problem of exactly matching the upper and lower bounds has remained open for the past six years.  The following theorem closes the issue for the count of triangles by proving that the lower bound is sharp.
\begin{thm}\label{theorem}
Let $T$ be the number of triangles in an Erd\H{o}s-R\'enyi graph $G(n,p)$. For each $\ep >0$ there are positive constants $C_1(\ep)$, $C_2(\ep)$ and $C_3(\ep)$ such that whenever $C_1(\ep) n^{-1}\log n\le p\le C_2(\ep)$, we have
\[
\pp(T\ge (1+\ep) \ee(T) ) \le e^{-C_3(\ep) n^2 p^2\log (1/p)}. 
\] 
\end{thm}
The proof does not involve martingales or counting arguments. It is somewhat hard to decide on a nomenclature for the method; it may be tentatively called it a `localization argument'. It is plausible that the technique may be useful in a wider class of problems. 

It must be mentioned that a couple of months after the first draft of this paper was posted on arXiv, DeMarco and Kahn \cite{demarcokahn10} posted a different proof of the conjecture. The DeMarco-Kahn proof is shorter and   gives a little more. %They also give sharp bounds in the regime $p < n^{-1}\log n$. 

It should also be mentioned that recently, a related problem has been investigated for dense graphs (i.e.\ $p$ fixed and $n \ra\infty$). The objective is to find the exact constant $C(p,\ep)$ such that
\[
\pp(T \ge (1+\ep)\ee(T)) = e^{-n^2 C(p, \ep)(1+o(1))}.
\]
In the language of large deviations, this is the problem of evaluating the large deviation rate function. The first progress in this problem was made in~\cite{chatterjeedey09} where it was shown that  given $p\in (0,1)$, there exist $p^3/6< t'\le t''< 1/6$ such that for all $t\in (p^3/6,t')\cup(t'',1/6)$, 
\begin{align*}\label{cd}
\pp(T \ge tn^3) = e^{-n^2I_p((6t)^{1/3})(1+o(1))}, 
\end{align*}
where $I_p$ is the function 
\[
I_p(x) := \frac{x}{2}\log\frac{x}{p} + \frac{1-x}{2}\log \frac{1-x}{1-p}. 
\]
However the result does not cover all values of $(p,t)$. The rate function in the full regime has been obtained very recently in \cite{chatterjeevaradhan10} as a consequence of a general large deviation principle for dense Erd\H{o}s-R\'enyi graphs. Two other papers in this direction are \cite{bolthausenetal09} and~\cite{doringeichelsbacher09}. 

\section{Notation and terminology}
In the following, $G$ will denote an Erd\H{o}s-R\'enyi graph $G(n,p)$ and $T$ will be the number of triangles in $G$. By convention, $C$ will denote any positive absolute constant, whose value may change from line to line. Similarly, $C(\ep)$ will denote any positive constant whose value depends only on $\ep$. 
For simplicity of notation, set  
\[
L := \log (1/p), \ \ \ell := 1/L.
\]
Given $\ep$, $n$ and $p$, let us call an edge in $G$  `good' if there are less than $\ep \ell np$ triangles containing the edge. Any edge that is not `good' will be called `bad'.
In the same vein, call a vertex `good' if it has less than $7np$ neighbors in $G$, and `bad' otherwise. (There is nothing special about $7$; any large enough constant is good for our purposes.) Define:
\begin{align*}
T' &:= \#\text{triangles in $G$ with all good edges.}\\
T_0 &:= \#\text{triangles in $G$ with at least one bad edge, but all good vertices.}\\
T_1 &:= \#\text{triangles in $G$ with exactly one bad vertex and two good vertices.}\\
T_2 &:= \#\text{triangles in $G$ with exactly two bad vertices and one good vertex.}\\
T_3 &:= \#\text{triangles in $G$ with all bad vertices.}
\end{align*}
Then clearly,
\begin{equation}\label{mainineq}
T \le T' + T_0 + T_1 + T_2 + T_3. 
\end{equation}
Thus, it suffices to get upper tail estimates for the summands on the right hand side. This is the program for  the subsequent sections.
Some further notation and terminology will be introduced along the way. 

\section{Concentration inequality}
The following theorem is most crucial component of the proof. It gives a generic concentration inequality for sums of dependent random variables. 
\begin{thm}\label{mainthm}
Let $F$ be a finite set and $(X_i)_{i\in F}$, $(X'_i)_{i\in F}$, $(X_{j(i)})_{i,j\in F}$ be collections of nonnegative random variables with finite moment generating functions, defined on the same probability space, and satisfying the following conditions:
\begin{enumerate}
\item[$(a)$] For all $i$, $X_i \le X'_i$.
\item[$(b)$] For all $i$, the random variables $X'_i$  and $\sum_{j\in F} X_{j(i)}$ are independent.%is independent of $(X_{j(i)})_{j\in F}$. 
%\item[$(c)$] For all $i,j$, $X_{j(i)} \le X_j$. 
\item[$(c)$] For all $i$, $\sum_{j\in F}X_{j(i)} \le \sum_{j\in F}X_j$. 
\item[$(d)$] There is a constant $a$ such that for all $i$, when $X_i > 0$, we have
\[
\sum_{j\in F} X_j \le a + \sum_{j\in F} X_{j(i)}. 
\]
\end{enumerate}
Let $\lambda := \sum_{i\in F} \ee(X'_i)$. Then for any $t\ge \lambda$,
\[
\pp\Big(\sum_{i\in F} X_i \ge t\Big) \le \exp\Big(-\frac{t}{a}\Big(\log \frac{t}{\lambda} - 1 + \frac{\lambda}{t}\Big)\Big) \le \exp\Big(-\frac{t}{a} \log \frac{t}{3\lambda}\Big).
\]
(Note that the second bound holds trivially for $0< t< \lambda$ as well.) 
\end{thm}
It is a bit difficult to convey the meaning of this theorem. Hopefully, it will become more transparent as the technique is applied numerous times in the proof of Theorem \ref{theorem}. The following simple lemma is the first step in the proof of Theorem \ref{mainthm}. 
\begin{lmm}\label{maintail}
Suppose $X$ is a nonnegative random variable with a finite moment generating function and $\lambda$ and $a$ are positive constants such that for any $\theta \ge 0$, we have
\[
\ee(Xe^{\theta X}) \le \lambda e^{\theta a} \ee(e^{\theta X}).
\]
Then for any $t \ge \lambda$ we have
\[
\pp(X\ge t) \le \exp\Big(-\frac{t}{a}\Big(\log \frac{t}{\lambda} - 1 + \frac{\lambda}{t}\Big)\Big). 
\]
\end{lmm}
\begin{proof}
Let $m(\theta):= \ee(e^{\theta X})$ be the moment generating function of $X$. From the hypothesis of the theorem, we have
\[
\frac{d}{d\theta} \log m(\theta) \le \lambda e^{\theta a}. 
\]
Integrating, we get
\[
m(\theta)\le \exp\Big(\frac{\lambda(e^{\theta a} - 1)}{a}\Big). 
\]
Again, for any positive $\theta$ and $t$,
\[
\pp(X \ge t) \le e^{-\theta t} m(\theta).
\]
The proof is completed by taking $\theta = a^{-1}\log (t/\lambda)$. %(Note that the result is trivially true if $t\le \lambda$.)
\end{proof}
\begin{proof}[Proof of Theorem \ref{mainthm}]
Let $X:=\sum_{i\in F} X_i$. Then for any $\theta \ge 0$, by the condition $(d)$ and the nonnegativity of the $X_i$'s, we have
\begin{align*}
\ee(Xe^{\theta X}) &= \sum_{i\in F} \ee(X_i e^{\theta X})\\
&\le \sum_{i\in F} \ee(X_i e^{\theta a + \theta \sum_{j\in F} X_{j(i)}}). 
\end{align*}
But again, due to the nonnegativity of $X_i$ and the conditions $(a)$ and $(b)$, we have
\[
\ee(X_i e^{\theta \sum_{j\in F} X_{j(i)}}) \le \ee(X'_i e^{\theta \sum_{j\in F} X_{j(i)}}) = \ee(X'_i)\ee( e^{\theta \sum_{j\in F} X_{j(i)}}).
\]
Finally, by condition $(c)$ we get
\[
\ee( e^{\theta \sum_{j\in F} X_{j(i)}})\le \ee( e^{\theta X}).
\]
Combining the steps we get
\[
\ee(Xe^{\theta X}) \le \lambda e^{\theta a} \ee(e^{\theta X}). 
\]
The proof of the first inequality is completed by applying Lemma \ref{maintail}. To get the second (weaker) inequality, simply observe that
\[
\log \frac{t}{\lambda} - 1 + \frac{t}{\lambda} \ge \log \frac{t}{e\lambda} \ge \log \frac{t}{3\lambda}.
\]
This completes the proof. 
\end{proof}

\section{Tail bound for $T'$}
The objective of this section is to obtain a tail inequality for the number of triangles with all good edges. 
Let us begin by introducing some notation, beyond what has been already defined in previous sections. Let $V^{(2)}$ denote the set of all unordered pairs of distinct vertices $uv$ and let $V^{(3)}$ denote the set of all unordered triplets of distinct vertices $uvw$.  
For each $uv\in V^{(2)}$, let
\[
I_{uv} := 1_{\{\text{$uv$ is an edge in $G$}\}}.
\]
Next, let  
\[
Z_{uv} := 1_{\{\text{$uv$ is a good edge in $G$}\}}.
\] 
For each $uvw\in V^{(3)}$, let 
\[
Y_{uvw} := 1_{\{\text{$uvw$ is a triangle in $G$}\}} = I_{uv}I_{vw}I_{uw},
\]
and let 
\[ 
X_{uvw} :=  1_{\{\text{$uvw$ is a triangle with all good edges}\}} = Z_{uv}Z_{vw}Z_{uw}.
\]
Then note that
\[
T' = \sum_{uvw\in V^{(3)}} X_{uvw}.
\]
We intend to apply Theorem \ref{mainthm} to the collection $(X_{uvw})$. We already have $Y_{uvw}$ such that $X_{uvw} \le Y_{uvw}$ and therefore condition $(a)$ is satisfied if we set $X'_{uvw} = Y_{uvw}$. The challenge is now to construct $X_{xyz(uvw)}$ appropriately. The rest of this section is devoted to the proof of the following result.
\begin{prop}\label{tp}
There exists a positive constant $C(\ep)$ depending only on $\ep$  such that  
\[
\pp(T' \ge \ee(T) +  \ep n^3 p^3) \le e^{-C(\ep)n^2 p^2 \log (1/p)}.
\]
\end{prop}
\begin{proof}
Fix a triplet $uvw\in V^{(3)}$. Let $K$ be the set of triplets that share no vertices with $uvw$ and let $K'$ be the set of triplets that share exactly one vertex with $uvw$. Let $K''$ be the set of triplets that share two or more vertices with $uvw$. Then $V^{(3)} = K \cup K' \cup K''$. 

Take any $xyz\in K'$. Then $xyz$ shares exactly one vertex with $uvw$.  Suppose without loss of generality  that $x=u$. Define the random variable $X_{xyz(uvw)}$ as follows. First, let
\begin{align*}
T_y &:= \#\text{triangles of the form $xyr$ where $r\not \in \{v,w\}$}, \\
T_z &:= \#\text{triangles of the form $xzr$ where $r\not \in \{v,w\}$}.  
\end{align*}
Next, let
\begin{align*}
N_y &:= \#\text{neighbors of $y$ in the set $\{v,w\}$},\\
N_z &:= \#\text{neighbors of $z$ in the set $\{v,w\}$}.
\end{align*}
Define
\begin{align*}
E_y &:= 1_{\{N_y + T_y < \ep\ell np\}}, \ \ E_z := 1_{\{N_z + T_z < \ep \ell np\}}.
\end{align*}
Finally, let
\[
X_{xyz(uvw)} := Y_{xyz}Z_{yz} E_y E_z. 
\]
Now note that since $\{v,w\} \cap \{y,z\} =\emptyset$, the definitions of $Y_{xyz}$, $Z_{yz}$, $T_y$, $T_z$, $N_y$ and $N_z$ do not involve the edges $uv, vw, uw$. In particular, the definition of $X_{xyz(uvw)}$ does not involve the edges $uv, vw, uw$. This is true for any $xyz\in K'$. 

If $xyz\in K$, let $X_{xyz(uvw)} = X_{xyz}$. Then clearly the definition of $X_{xyz(uvw)}$ does not involve the edges $uv, vw, uw$.

Finally, if $xyz\in K''$, let $X_{xyz(uvw)} \equiv 0$.

From the above construction and observations,  we see that condition $(b)$ of Theorem \ref{mainthm} is satisfied.

Next, suppose that for some $xyz\in K'$ with $x=u$, we have $X_{xyz(uvw)} = 1$ in a particular realization of $G$. Then $xyz$ is a triangle, $yz$ is a good edge, and $E_y = E_z = 1$. We claim that $xy$ is a good edge. To see this, simply note that the number of triangles containing the edge $xy$ is bounded by $T_y + N_y$. Similarly, $xz$ is also a good edge. Therefore $X_{xyz} = 1$. Thus, for $xyz\in K'$, $X_{xyz(uvw)} \le X_{xyz}$. When $xyz\in K\cup K''$, the inequality is trivially true. This establishes condition $(c)$ of Theorem \ref{mainthm}.

Finally, suppose that for some $xyz\in K'$ with $x=u$, we have $X_{xyz} = 1$ in a particular realization of $G$, which also satisfies $X_{uvw} = 1$. Then $xv$ and $xw$ are edges in $G$, and hence the number of triangles containing the edge $xy$ is exactly equal to $N_y + T_y$. Therefore in this situation, $E_y = Z_{xy}$. Similarly $E_z = Z_{xz}$. Therefore, $X_{xyz(uvw)} = Y_{xyz}Z_{xy}Z_{xz}Z_{yz} = X_{xyz}$. %Thus, when $X_{uvw} =1$, we have $X_{xyz(uvw)} \ge X_{xyz}$ for all $xyz\in K'$. However, as we proved in the earlier paragraph, we always have $X_{xyz(uvw)} \le X_{xyz}$. 
Thus, when $X_{uvw} = 1$, we have $X_{xyz(uvw)} = X_{xyz}$ for all $xyz\in K'$. The same is trivially true for $xyz\in K$.

Moreover, if $X_{uvw} = 1$, the number of triangles in $G$  sharing at least two vertices with $uvw$  is bounded by $3\ep \ell np$. Combining these observations,  we see that when $X_{uvw} = 1$,
\[
\sum_{xyz\in V^{(3)}} X_{xyz} \le 3\ep \ell np + \sum_{xyz\in V^{(3)}} X_{xyz(uvw)}. 
\]
This establishes condition $(d)$ of Theorem \ref{mainthm} with $a = 3\ep \ell np$. Thus, if we let $\lambda := \sum_{xyz} \ee(X'_{xyz}) = \ee(T)$, then by the first inequality in Theorem \ref{mainthm} we get
\[
\pp(T' \ge \lambda + \ep n^3 p^3) \le \exp\Big(-\frac{\lambda + \ep n^3 p^3}{\ep \ell np} (\log c - 1 + c^{-1})\Big),
\]
where 
\[
c = c(n,p, \ep) = \frac{\lambda + \ep n^3 p^3}{\lambda}. 
\]
It is easy to check that $\log c - 1 + c^{-1}$ converges to a positive constant depending only on $\ep$ as $np \ra \infty$. This completes the proof. 
\end{proof}

\section{Tail bound for $T_0$}
The purpose of this section is to obtain a tail bound for the number of triangles in $G$ with at least one bad edge but all good vertices. We shall continue to use the notation and terminology introduced in the previous sections. We also need to introduce some additional notation. For each $uv\in V^{(2)}$, let 
\[
t_{uv} := \sum_{w\in V\backslash \{u,v\}} I_{uw} I_{vw}.
\]
For a set $F\subseteq V^{(2)}$, let
\[
t(F) := \sum_{uv\in F} t_{uv}.
\]
Two elements of $V^{(2)}$ will be called `non-adjacent' if they do not share a common vertex, and `adjacent' otherwise. A set $F \subseteq V^{(2)}$ will be called `matching' if no two elements of $F$ share a common vertex.
\begin{lmm}\label{e4lmm}
Let $A$ be a matching. Then for each $t > 0$,
\[
\pp(t(A) \ge t) \le \exp\Big(-\frac{t}{3}\log \frac{t}{3|A|np^2}\Big). 
\]
\end{lmm}
\begin{proof}
Let $R$ be the set of triplets $uvw$ such that $uv\in A$ and $w\in V\backslash \{u,v\}$. Throughout this proof, $uvw$ will denote a typical element of $R$. For each such $uvw$, let
\[
W_{uvw} := I_{uw} I_{vw},
\]
Since $t(A) = \sum_{uvw\in R} W_{uvw}$, We intend to apply Theorem \ref{mainthm} to the collection $(W_{uvw})_{uvw\in R}$. 

Fix $uvw\in R$. Let $D$ be the set of all $xyz \in R$ such that 
\begin{equation}\label{nonemp}
\{xz, yz\} \cap \{uw, vw\} \ne \emptyset.
\end{equation}
This can happen if $uvw = xyz$. Suppose this is not the case. Then since $A$ is a matching, $\{u,v\}\cap \{x,y\} =\emptyset$.  Thus if \eqref{nonemp} holds and $uvw \ne xyz$, then we must have $\{x,y\}\cap\{ w\} \ne \emptyset$ and $\{u,v\}\cap \{z\}\ne \emptyset$. But again because $A$ is a matching, the choice of $x$ determines the choice of $y$. Combining the above observations, we see that $|D|\le 3$. 

Since $W_{uvw}$ depends only on the edges $uw$ and $vw$, $W_{uvw}$ is independent of the collection  $\{W_{xyz}\}_{xyz\in R\backslash D}$. So, if we define $W_{xyz(uvw)} = 0$ for $xyz \in D$, $W_{xyz(uvw)} = W_{xyz}$ for $xyz \in R \backslash D$, and $W'_{xyz}=W_{xyz}$ for all $xyz\in R$, then the conditions of Theorem \ref{mainthm} are satisfied with $a = 3$ and $\lambda = \sum_{xyz\in R} \ee(W_{xyz}') \le |A|np^2$. This completes the proof of the lemma.
\end{proof}
Next, let us define two events:
\begin{align*}
E_1 &:= \{\text{$\exists \, F \subseteq V^{(2)}$ such that $F$ is a matching, $|F|> Lnp$,}\\
&\qquad \qquad \text{ and all elements of $F$ are bad edges in $G$}\}.\\
E_2 &:= \{\text{$\exists \, F \subseteq V^{(2)}$ such that $F$ is a matching, $|F|\le Lnp$,}\\
&\qquad \qquad \text{ and $t(F) > \ep n^2 p^2$}\}. 
\end{align*}
\begin{lmm}\label{e4}
There is a constant $C(\ep)>0$ depending only on $\ep$ such that whenever $C(\ep)^{-1} n^{-1}\log n\le p\le C(\ep)$, we can conclude that both $\pp(E_1)$ and $\pp(E_2)$ are bounded by $\exp(-C(\ep)n^2 p^2 \log (1/p))$. 
\end{lmm}
\begin{proof}
If $E_1$ holds, there exists a set $F$ satisfying the conditions for $E_1$. By arbitrarily dropping some elements of $F$, we can find a subset $F'$ of $F$ such that $Lnp < |F'|\le 2Lnp$. Then $F'$ is again a matching, and  since the elements of $F'$ are bad edges in $G$, 
\[
t(F') \ge |F'|\ep\ell np \ge \ep n^2 p^2. 
\]
By Lemma \ref{e4lmm} we know that for any fixed matching  $A\subseteq V^{(2)}$ of size $\le 2Lnp$, 
\[
\pp(t(A) \ge \ep n^2 p^2)\le \exp\Big(-\frac{\ep n^2p^2}{3}\log \frac{\ep n^2 p^2}{6 L n^2 p^3}\Big). 
\]
Clearly, there exists $C(\ep)>0$ such that if $p\le C(\ep)$, the expression of the right is bounded by $ e^{- C(\ep)n^2 p^2 \log (1/p)}$. 
The number of choices of $A\subseteq V^{(2)}$ with $|A|\le 2Lnp$ is bounded by 
\[
\sum_{0\le k\le 2Lnp} {n(n-1)/2 \choose k} \le e^{CLnp \log n}. 
\]
Combining the above observations, we see that
\[
\pp(E_1) \le  e^{CLnp \log n - C(\ep)n^2 p^2 \log (1/p)}.
\]
Thus, if $C'(\ep)^{-1}n^{-1}\log n\le p\le C'(\ep)$ for some appropriately small constant $C'(\ep)$, we get the required bound.

To get the bound on $\pp(E_2)$, we similarly apply Lemma \ref{e4lmm} with fixed $A\subseteq V^{(2)}$ of size $\le Lnp$ and $t = \ep n^2 p^2$, and then take union bound over all choices of $A$. 
\end{proof}

Finally we arrive at the main result of this section.
\begin{prop}\label{t0}
There is a constant $C(\ep)>0$ depending only on $\ep$ such that if $C(\ep)^{-1} n^{-1}\log n\le p\le C(\ep)$, we have 
\[
\pp(T_0 > 15\ep n^3 p^3)\le e^{-C(\ep) n^2 p^2 \log (1/p)}.
\]
\end{prop}
\begin{proof}
Let $B'$ be the set of bad edges with both endpoints in the set of good vertices. We shall now show that if $E_1$ and $E_2$ are both false, then $t(B') \le 15 \ep n^3 p^3$. This will complete the proof of the lemma, since $T_0 \le t(B')$. 

Recall that we call two elements of $B'$ `adjacent' if they share a common vertex in $G$. This defines an undirected graph structure on $B'$. Since the degree of an endpoint of any element of $B'$ (in the graph $G$) is less than $7np$, it is clear that the maximum vertex degree of the adjacency graph on $B'$ is less than $14np$. Thus, there is a coloring of  this graph with $\le 14np + 1\le 15 np$ colors such that no two adjacent elements of $B'$ receive the same color. (This is a standard argument in graph theory: arrange the elements of $B'$ in some arbitrary order; color the $i$th element with a color that was not given to any of its neighbors among the first $i-1$ elements. This produces a coloring such that no two adjacent elements receive the same color, and it is also clear that we need at most $14np + 1$ colors.)%  (see e.g.\ Diestel \cite{diestel00}, page 98). 

Let us fix such a coloring. For each color $c$, let $F_c$ denote the subset of $B'$ that receives the color $c$. Note that each $F_c$ is a matching by construction.

Now take the color $c$ that maximizes $t(F_c)$. Then 
\[
t(F_c) \ge \frac{t(B')}{\text{number of colors}} \ge \frac{t(B')}{15 np}. 
\] 
By the falsity of $E_1$ we have $|F_c|\le Lnp$; and therefore, the falsity of $E_2$ shows that $t(F_c)\le \ep n^2 p^2$. Thus, $t(B') \le 15 \ep n^3 p^3$.  This completes the proof. 
\end{proof}

\section{Tail bound for $T_1$}
In this section we bound the probability that there are too many triangles with exactly one bad vertex and two good vertices. As usual, we will continue to adhere to the notation and terminology introduced in the preceding sections. Additionally, let us define the event 
\[
E_3 := \{\text{There are more than $Lnp$ bad vertices}\}.
\]
For each $u\in V$, let $d_u$ be the degree of $u$ in $G$. For each $A\subseteq V$ let
\[
d(A) := \sum_{u\in A} d_u.
\]
\begin{lmm}\label{bintail}
Let $X$ be a binomial random variable with mean $\lambda$. Then for any $t > 0$, 
\[
\pp(X \ge t) \le e^{-t\log (t/3\lambda)}.
\]
\end{lmm}
\begin{proof}
Since $X$ is a binomial random variable, it can be expressed as $\sum X_i$, where $X_i$ are i.i.d.\ Bernoulli random variables. Setting $X_i' = X_i$, and $X_{j(i)}= X_j$ if $j\ne i$ and $X_{i(i)} = 0$, we are in the setting of Theorem \ref{mainthm} with $a=1$ and $\lambda = \ee(X)$. This completes the proof.
\end{proof}
\begin{lmm}\label{degtail}
For any set of vertices $A$, and any $t\ge 0$, we have
\[
\pp(d(A)\ge t)\le \exp\Big(-\frac{t}{2}\log \frac{t}{6n|A|p}\Big). 
\]
\end{lmm}
\begin{proof}
Let $Y$ be the number of distinct edges in $G$ with at least one endpoint in $A$. A simple argument shows that $d(A) \le 2Y$. Note that $Y$ is a binomial random variable with mean $\le n|A|p$. Thus by Lemma \ref{bintail}, the proof is done.
\end{proof} 

\begin{lmm}\label{e3lmm}
There is an absolute constant $C$ such that if $p > C^{-1}n^{-1}\log n$, we have 
\[
\pp(E_3)\le e^{-Cn^2 p^2 \log(1/p)}.
\]
\end{lmm}
\begin{proof}
For any set of vertices $A$ with $|A|\le \lceil Lnp\rceil $, Lemma \ref{degtail} gives that
\begin{align*}
\pp(d(A) \ge 7L n^2 p^2) &\le \exp\Big(-\frac{7Ln^2p^2}{2} \log \frac{7Ln^2 p^2}{6\lceil Lnp \rceil np}\Big)\\
&\le e^{-CLn^2 p^2}.
\end{align*}
The number of choices of $A$ with $|A|\le \lceil Lnp\rceil $ is bounded by $e^{CLnp \log n}$. Thus, if we define 
\[
E_4 := \{\text{$\exists\; A\subseteq V$ with $|A|\le \lceil Lnp\rceil $ and $d(A) \ge 7Ln^2 p^2$}\},
\]
then there are absolute constants $C'$ and $C''$ such that 
\begin{align*}
\pp(E_4) &\le \exp(C'Lnp\log n - C''L n^2p^2). 
\end{align*}
So, if $p\ge C^{-1}n^{-1}\log n$ for some suitable constant $C$, we have
\[
\pp(E_4) \le e^{-Cn^2p^2 \log (1/p)}. 
\]
Next, note that if $E_3$ is true, then there is a set $A$ of vertices of size exactly $\lceil Lnp \rceil$, each of which has degree $\ge 7np$. Consequently $d(A) \ge 7np \lceil Lnp \rceil\ge 7Ln^2 p^2$. Thus, $E_3$ implies $E_4$. This completes the proof of the lemma.
\end{proof}
\begin{prop}\label{t1}
There is a constant $C(\ep)>0$ depending only on $\ep$ such that whenever $C(\ep)^{-1} n^{-1}\log n\le p\le C(\ep)$, we have
\[
\pp(T_1 \ge \ep n^3 p^3) \le e^{-C(\ep)n^2p^2 \log (1/p)}.
\]
\end{prop}
\begin{proof}
Fix a set $A\subseteq V$ of size $\le Lnp$. Let $R_1(A)\subseteq V^{(3)}$ be the set of  unordered triplets with exactly $1$ vertex from $A$ and $2$ vertices from $A^c$. Let $T_1(A)$ be the number of triangles in $G$ with exactly $1$ vertex from $A$ and $2$ vertices from the set of good vertices in $A^c$. 

A typical element of $R_1(A)$ will be written in the form $uvw$, where $u\in A$ and $vw$ is an unordered pair of vertices from $A^c$. For each $uvw\in R_1(A)$, let $P_{uvw}$ be the indicator that $uvw$ is a triangle in $G$ and $v,w$ are good vertices. Then
\[
T_1(A) = \sum_{uvw\in R_1(A)} P_{uvw}.
\]
We shall first compute a tail bound for $T_1(A)$. We intend to apply Theorem~\ref{mainthm} to the collection $(P_{uvw})_{uvw\in R_1(A)}$ for this purpose. As usual, let $P'_{uvw} = Y_{uvw}$, where $Y_{uvw}$ is the indicator that $uvw$ is a triangle in $G$. This clearly verifies condition $(a)$ of Theorem \ref{mainthm}.

Fix $uvw\in R_1(A)$, with $u\in A$ and $v,w\in A^c$. Let $K\subseteq R_1(A)$ be the set of triplets that share no vertices with $uvw$ and let $K'\subseteq R_1(A)$ be the set of triplets that share exactly one vertex with $uvw$. %, and let $K''$ be the set of triplets that share exactly two vertices with $uvw$. 

Now take any $xyz\in K'$, with $x\in A$, $y,z\in A^c$. Then $xyz$ shares exactly one vertex with $uvw$.  Suppose the common vertex is in $A^c$. Then without loss of generality, $y=v$. Define $P_{xyz(uvw)}$ to be the indicator that $xyz$ is a triangle in $G$, that $z$ is a good vertex, and that the number of neighbors of $y$ in $V\backslash\{u,w\}$ is $< 7np-2$. %Similarly, if $z = w$, we define $P_{xyz(uvw)}$ to be the indicator that $xyz$ is a triangle, that $y$ is a good vertex, and that the number of neighbors of $z$ in $V\backslash\{u,v\}$ is $< 7np-2$. 
If the common vertex is in $A$, that is $x=u$, let $P_{xyz(uvw)} = P_{xyz}$. 

If $xyz\in K$, let $P_{xyz(uvw)} = P_{xyz}$. If $xyz$ shares two or more vertices with $uvw$, let $P_{xyz(uvw)}= 0$. 

From the above construction, it is clear that the definitions of the random variables $\{P_{xyz(uvw)}\}_{xyz\in R_1(A)}$ do  not involve the edges $uv, vw, uw$, and therefore this collection is independent of $P'_{uvw}$. This verifies condition $(b)$ of Theorem \ref{mainthm}. 

Another easy verification shows that whenever $P_{xyz(uvw)} = 1$, we must have $P_{xyz}=1$. Thus $P_{xyz(uvw)} \le P_{xyz}$. This establishes condition $(c)$ of Theorem \ref{mainthm}.

Lastly, suppose that $P_{uvw} = 1$ in a particular realization of $G$. If $xyz\in K'$ and $P_{xyz} =1$ in that realization of $G$, then it is easy to see that we also have $P_{xyz(uvw)} = 1$. Moreover, since $P_{uvw} =1$, there can be at most $21np$ triangles sharing two or more vertices with $uvw$. Thus, if $P_{uvw} =1$, we have
\[
\sum_{xyz\in R_1(A)}P_{xyz} \le 21np + \sum_{xyz\in R_1(A)} P_{xyz(uvw)}. 
\]
Therefore, we have established that all the conditions of Theorem \ref{mainthm} hold, with $a=21np$ and $\lambda = \sum_{xyz\in R_1(A)}\ee(P'_{xyz}) \le n^2(Lnp) p^3$. This shows that for any $A\subseteq V$ with $|A|\le Lnp$, we have 
\[
\pp(T_1(A) \ge \ep n^3 p^3) \le e^{-C(\ep)Ln^2p^2}, 
\]
provided $p\le C'(\ep)$ for some small enough constant $C'(\ep)$. 

Now, if $T_1 \ge \ep n^3p^3$, then either $E_3$ happens, or there exists a set $A\subseteq V$ (of bad vertices) with $|A|\le Lnp$ and $T_1(A) \ge \ep n^3 p^3$. The number of choices of such $A$ is bounded by $e^{CLnp\log n}$. Thus, there are constants $C'$ and $C''(\ep)$ such that 
\begin{align*}
\pp(T_1 \ge \ep n^3 p^3) &\le \pp(E_3) + e^{C'Lnp\log n - C''(\ep)n^2 p^2 \log (1/p)}. 
\end{align*}
If $p> C(\ep)^{-1} n^{-1} \log n$ for some appropriately small constant $C(\ep)$, then we can combine the above bound with the bound on $\pp(E_3)$ from Lemma \ref{e3lmm} to complete the proof of the Proposition. 
\end{proof}

\section{Tail bound for $T_2$}
In this section we bound the probability that there are too many triangles with exactly two bad vertices and one good vertex. As usual, we will continue to adhere to the notation and terminology introduced in the preceding sections. We prove the following analog of Proposition \ref{t1}. The method of proof is similar.
\begin{prop}\label{t2}
There is a constant $C(\ep)>0$ depending only on $\ep$ such that whenever $C(\ep)^{-1} n^{-1}\log n\le p\le C(\ep)$, we have
\[
\pp(T_2 \ge \ep n^3 p^3) \le e^{-C(\ep)n^2p^2 \log (1/p)}.
\]
\end{prop}
\begin{proof}
Fix a set $A\subseteq V$ of size $\le Lnp$. Let $R_2(A)\subseteq V^{(3)}$ be the set of unordered triplets with exactly $2$ vertices from $A$ and $1$ vertex from $A^c$. Let $T_2(A)$ be the number of triangles in $G$ with exactly $2$ vertices from $A$ and $1$ vertex from the set of good vertices in $A^c$. 

A typical element of $R_2(A)$ will be written in the form $uvw$, where $uv$ is an unordered pair of vertices from $A$ and $w\in A^c$. For each $uvw\in R_2(A)$, let $Q'_{uvw}$ be the indicator that $uw$ and $vw$ are edges in $G$, and let $Q_{uvw}$ be the indicator that $uw$, $vw$ are edges in $G$ and $w$ is a  good vertex. Then 
\[
T_2(A) \le S(A) := \sum_{uvw\in R_2(A)}Q_{uvw}. 
\]
We intend to apply Theorem \ref{mainthm} to the collection $(Q_{uvw})_{uvw\in R_2(A)}$ to obtain a tail inequality for $S(A)$. Clearly $Q_{uvw} \le Q'_{uvw}$ and hence condition $(a)$ of Theorem \ref{mainthm} is satisfied. 

Now fix $uvw\in R_2(A)$, with $u,v\in A$ and $w\in A^c$. Let $K\subseteq R_2(A)$ be the set of triplets $xyz$ such that $z\ne w$. Let $K'$ be the set of triplets $xyz\in R_2(A)$ such that $z=w$ but $\{u,v\} \cap \{x,y\} =\emptyset$. Let $K''$ be the set of $xyz\in R_2(A)$ such that $z=w$ and $\{x,y\}\cap \{u,v\}\ne \emptyset$. Note that $R_2(A) = K \cup K' \cup K''$. 

Take any $xyz\in K'$. Define $Q_{xyz(uvw)}$ to be the indicator that $xz$ and $yz$ are edges in $G$, and that the number of neighbors of $z$ in $V\backslash\{u,v\}$ is $< 7np-2$. 

If $xyz\in K$, let $Q_{xyz(uvw)} = Q_{xyz}$. Lastly, if $xyz\in K''$, let $Q_{xyz(uvw)} = 0$. 

As in the proof of Proposition \ref{t1}, it is clear by construction that the definitions of the random variables $\{Q_{xyz(uvw)}\}_{xyz\in R_2(A)}$ do  not involve the edges $uw$ and $vw$, and therefore this collection is independent of $Q'_{uvw}$. This verifies condition $(b)$ of Theorem \ref{mainthm}. 

Again, if $Q_{xyz(uvw)} = 1$, we must have $Q_{xyz}=1$. Thus $Q_{xyz(uvw)} \le Q_{xyz}$, which establishes condition $(c)$ of Theorem \ref{mainthm}.

Finally, suppose that $Q_{uvw} = 1$ in a particular realization of $G$. If $xyz\in K'$ and $Q_{xyz} =1$ in that realization of $G$, then we also have $Q_{xyz(uvw)} = 1$. Moreover, since $Q_{uvw} =1$, we can easily conclude that the number of $xyz\in K''$ such that $Q_{xyz} = 1$ is bounded by $14 np$. Thus, if $Q_{uvw} =1$, we have
\[
\sum_{xyz\in R_2(A)}Q_{xyz} \le 14np + \sum_{xyz\in R_2(A)} Q_{xyz(uvw)}. 
\]
Therefore, we have shown that all conditions of Theorem \ref{mainthm} hold, with $a=14np$ and $\lambda = \sum_{xyz\in R_2(A)} \ee(Q'_{xyz}) \le n (Lnp)^2 p^2$. This shows that for any $A\subseteq V$ with $|A|\le Lnp$, we have 
\[
\pp(T_2(A) \ge \ep n^3 p^3) \le\pp(S(A) \ge \ep n^3 p^3) \le  e^{-C(\ep)Ln^2p^2}, 
\]
provided $p\le C'(\ep)$ for some small enough constant $C'(\ep)$.

We can now proceed exactly as in the last part of the proof of Proposition~\ref{t1} to complete the proof by invoking Lemma \ref{e3lmm}.
\end{proof}

\section{Tail bound for $T_3$}
In this section we obtain a tail bound for the number of triangles with all bad vertices. First, we need a simple lemma about the minimum number of edges in a graph with  given number triangles. The content of the lemma can be derived as a corollary of Theorem A of Alon \cite{alon81}, but the result is so simple that we give a direct proof. 
\begin{lmm}\label{anticauchy}
For any real symmetric matrix $(a_{ij})_{i,j=1}^n$, we have
\[
\sum_{i,j=1}^n a_{ij}^2 \ge \biggl|\sum_{i,j,k=1}^n a_{ij} a_{jk} a_{ki}\biggr|^{2/3}.
\]
Consequently, if an undirected graph has $r$ triangles, then it must have at least $\frac{1}{2}(6r)^{2/3}$ edges.
\end{lmm}
\begin{proof}
By two applications of the Cauchy-Schwarz inequality we get
\begin{align*}
\biggl|\sum_{i,j,k=1}^n a_{ij} a_{jk} a_{ki}\biggr| &\le \biggl(\sum_{i,j=1}^n a_{ij}^2\biggr)^{1/2} \biggl(\sum_{i,j=1}^n \biggl(\sum_{k=1}^n a_{ik} a_{jk}\biggr)^2\biggr)^{1/2}\\
&\le \biggl(\sum_{i,j=1}^n a_{ij}^2\biggr)^{1/2}\biggl(\sum_{i,j=1}^n\sum_{k,l=1}^n a_{ik}^2 a_{jl}^2\biggr)^{1/2} \\
&= \biggl(\sum_{i,j=1}^n a_{ij}^2\biggr)^{3/2}.
\end{align*}
This completes the proof of the first assertion of the theorem. For the graph theoretic conclusion, simply take $(a_{ij})$ to be adjacency matrix of the graph.
\end{proof}

\begin{prop}\label{t3}
There is a constant $C(\ep)>0$ depending only on $\ep$  such that whenever $C(\ep)^{-1} n^{-1}\log n\le p\le C(\ep)$, we have
\[
\pp(T_3 \ge \ep n^3 p^3) \le e^{-C(\ep)n^2p^2 \log (1/p)}.
\]
\end{prop}
\begin{proof}
Take any set $A\subseteq V$ with $|A|\le Lnp$. Let $T_3(A)$ be the number of triangles with all vertices in $A$. If $T_3(A) \ge \ep n^3 p^3$, then by Lemma \ref{anticauchy}, the number of edges in $A$ must be $\ge C(\ep) n^2 p^2$. The probability of this event is bounded by
\begin{align}\label{t3ineq}
{|A|(|A|-1)/2 \choose \lceil C (\ep)n^2 p^2\rceil } p^{C(\ep) n^2 p^2 }. 
\end{align}
For any two integers $1\le b< a$, it is simple to see that 
\[
{a \choose b} \le \frac{a^b}{b!} \le \biggl(\frac{ae}{b}\biggr)^b = e^{b + b\log (a/b)}. 
\]
Applying this bound, we see that the quantity \eqref{t3ineq} is bounded by 
\[
\exp\Big(C'(\ep)n^2 p^2 + C'(\ep) n^2 p^2 \log \frac{(Lnp)^2}{\lceil C(\ep)n^2 p^2\rceil } - C''(\ep)n^2 p^2 \log (1/p)\Big),
\]
where $C(\ep)$, $C'(\ep)$ and $C''(\ep)$ are constants depending only on $\ep$. 
If $p$ is sufficiently small (depending on $\ep$), the above quantity  is less than $e^{-C(\ep) L n^2 p^2}$.
The proof is now completed by using Lemma \ref{e3lmm} as in the last part of the proof of Proposition~\ref{t1}. 
\end{proof}

\section{Proof of Theorem \ref{theorem}}
Combining Propositions \ref{tp}, \ref{t0}, \ref{t1}, \ref{t2}, \ref{t3}, and the inequality \eqref{mainineq}, it is easy to see that there is a constant $C(\ep)>0$ depending only on $\ep$, such that for  $C(\ep)^{-1}n^{-1}\log n\le p\le C(\ep)$, 
\[
\pp(T>\ee(T)+ 19\ep n^3 p^3) \le e^{-C(\ep) n^2 p^2 \log (1/p)}. 
\]
This completes the proof. 

\vskip.2in
\noindent{\bf Acknowledgements.}
The author is indebted to Amir Dembo for bringing the problem to his attention and to the referees for very careful reports. The author thanks Partha Dey and Will Perkins for checking the proof, and Nick Crawford for useful discussions in the early stages of the project.

\end{document}